\documentclass[a4paper,10pt]{article}
\usepackage{amssymb}
\usepackage{graphicx}
\usepackage{bbm}
\usepackage{mathrsfs}
\usepackage{amsmath}
\usepackage{amsthm}
\usepackage{verbatim}
\usepackage{authblk}

\newtheoremstyle{mytheorem}{}{}{\itshape}{}{\bfseries}{:}{\newline}{}
\newtheoremstyle{mydefinition}{}{}{}{}{\bfseries}{:}{\newline}{}
\newtheoremstyle{myproof}{}{}{}{}{\bfseries}{:}{\newline}{#1#3}

\theoremstyle{mytheorem}
\newtheorem{thm}{Theorem}[]

\newtheorem{cor}[thm]{Corollary}
\newtheorem{lem}[thm]{Lemma}
\newtheorem{prop}[thm]{Proposition}

\theoremstyle{mydefinition}

\theoremstyle{myproof}

\newcommand{\Pb}{\mathbb{P}}
\newcommand{\Eb}{\mathbb{E}}
\newcommand{\Rb}{\mathbb{R}}
\newcommand{\Pt}{\tilde{\mathbb{P}}}
\newcommand{\Qt}{\tilde{\mathbb{Q}}}
\newcommand{\Fg}{\mathcal{F}}
\newcommand{\Ft}{\tilde{\mathcal{F}}}
\newcommand{\Gg}{\mathcal{G}}
\newcommand{\Gt}{\tilde{\mathcal{G}}}
\newcommand{\hs}{\hspace{2mm}}
\newcommand{\hsl}{\hspace{1mm}}
\newcommand{\ind}{\mathbbm{1}}

\title{Branching Brownian motion:\\ Almost sure growth along scaled paths}

\author[1]{S.C.~Harris}

\author[2]{M.I.~Roberts}

\affil[1]{Department of Mathematical Sciences, University of Bath, Claverton Down, Bath, BA2 7AY. \emph{E-mail:} S.C.Harris@bath.ac.uk}
\affil[2]{Laboratoire de Probabilit\'es et Mod\`eles Al\'eatoires, Universit\'e Paris VI, 175 rue du Chevaleret, 75013 Paris. \emph{E-mail:} matthew.roberts@upmc.fr}

\begin{document}

\maketitle

\begin{abstract}
\noindent
We give a proof of a result on the growth of the number of particles along chosen paths in a branching Brownian motion. The work follows the approach of classical large deviations results, in which paths of particles in $C[0,T]$, for large $T$, are rescaled onto $C[0,1]$. The methods used are probabilistic and take advantage of modern spine techniques.
\end{abstract}

\section{Introduction and statement of result}

\subsection{Introduction}
Fix $r > 0$ and a random variable $A$ taking values in $\{2,3,\ldots\}$ such that \mbox{$m:=E[A]-1 >1$} and $E[A\log A]<\infty$. We consider a branching Brownian motion (BBM) under a probability measure $\Pb$, which is described as follows. We begin with one particle at the origin. Each particle $u$, once born, performs a Brownian motion independent of all other particles, until it dies, an event which occurs at an independent exponential time after its birth with mean $1/r$. At the time of a particle's death it is replaced (at its current position) by a random number $A_u$ of offspring where $A_u$ has the same distrubition as $A$. Each of these particles, relative to its initial position, repeats (independently) the stochastic behaviour of its parent.

We let $N(t)$ be the set of particles alive at time $t$, and for $u\in N(t)$ let $X_u(t)$ be the position of particle $u$ at time $t$. Fix a set $D\subseteq C[0,1]$ and $\theta\in[0,1]$; then we are interested in the size of the sets
\[N_T(D,\theta) := \{u \in N(\theta T) : \exists f \in D \hbox{ with } X_u(t) = Tf(t/T) \hs \forall t\in[0,\theta T]\}\]
for large $T$.

\subsection{The main result}\label{main_section}
We define the class $H_1$ of functions by
\[H_1 := \left\{ f \in C[0,1] : \exists g \in L^2[0,1] \hbox{ with } f(s) = \int_0^s g(s) ds \hs \forall s\in[0,1]\right\},\]
and to save on notation we set $f'(t):=\infty$ if $f\in C[0,1]$ is not differentiable at the point $t$. We then take integrals in the Lebesgue sense so that we may integrate functions that equal $\infty$ on sets of zero measure. We let
\[\theta_0(f):= \inf\left\{\theta\in[0,1] : rm\theta - \frac{1}{2}\int_0^\theta f'(s)^2 ds < 0\right\} \hsl \in [0,1]\cup\{\infty\}\]
(we think of $\theta_0$ as the extinction time along $f$, the time at which the number of particles near $f$ hits zero) and define our rate function $K$, for $f\in C[0,1]$ and $\theta\in[0,1]$, as
\[K(f,\theta) := \left\{\begin{array}{ll}rm\theta - \frac{1}{2}\int_0^\theta f'(s)^2 ds & \hbox{ if } f\in H_1 \hbox{ and } \theta\leq\theta_0(f)\\ -\infty & \hbox{ otherwise.}\end{array}\right.\]
We expect approximately $\exp(K(f,\theta)T)$ particles whose paths up to time $\theta T$ (when suitably rescaled) look like $f$. This is made precise in Theorem \ref{main_thm}.

\begin{thm}\label{main_thm}
For any closed set $D\subseteq C[0,1]$ and $\theta\in[0,1]$,
\[\limsup_{T\to\infty} \frac{1}{T}\log|N_T(D,\theta)|\leq \sup_{f\in D} K(f,\theta)\]
almost surely, and for any open set $U\subseteq C[0,1]$ and $\theta\in[0,1]$,
\[\liminf_{T\to\infty} \frac{1}{T}\log|N_T(U,\theta)|\geq \sup_{f\in U} K(f,\theta)\]
almost surely.
\end{thm}

\noindent
Sections \ref{lower_bound_section} and \ref{upper_bound_section} will be concerned with giving a proof of this theorem.

\vspace{2mm}

An almost identical result was stated by Git in \cite{git:almost_sure_path_properties}. We would like to give an alternative proof for two reasons.

Firstly, we believe that our proof of the lower bound is perhaps more intuitive, and certainly more robust, than that given in \cite{git:almost_sure_path_properties}. There are many more general setups for which our proofs will go through without too much extra work. One possibility is to allow particles to die without giving birth to any offspring (that is, to allow $A$ to take the value $0$): in this case the statement of the theorem would be conditional on the survival of the process, and we will draw attention to any areas where our proof must be adapted significantly to take account of this. There is work in progress on some further interesting cases and their applications, in particular the case where breeding occurs at the inhomogeneous rate $r x^p$, $p\in[0,2)$, for a particle at position $x$.

Secondly, there seems to be a slight oversight in the proof of Lemma 1 in \cite{git:almost_sure_path_properties}, and that lemma is then used in obtaining both the upper and lower bounds. Although the gap seems minor at first, the complete lack of simple continuity properties of the processes involved means that almost all of the work involved in proving the upper bound is concerned with this matter. We give details of the oversight as an appendix.

Our tactic for the proof is to first work along lattice times, and then upgrade to the full result using Borel-Cantelli arguments. We begin, in Section \ref{spine_section}, by introducing a family of martingales and changes of measure which will provide us with intuitive tools for our proofs. We then apply these tools to give an entirely new proof of the lower bound for Theorem \ref{main_thm} in Section \ref{lower_bound_section}. Finally, in Section \ref{upper_bound_section}, we take the same approach as in \cite{git:almost_sure_path_properties} to gain the upper bound along lattice times, and then rule out some technicalities in order to move to continuous time.

This work complements the article by Harris \& Roberts \cite{harris_roberts:unscaled_growth}. Large deviation probabilities for the same model were given by Lee \cite{lee:large_deviations_for_branching_diffusions} and Hardy \& Harris \cite{hardy_harris:spine_bbm_large_deviations}.

\section{A family of spine martingales}\label{spine_section}

\subsection{The spine setup}

We will need to use some modern ``spine'' techniques as part of our proof. We only need some of the most basic spine tools, and we do not attempt to explain the details of these rigourously, but rather refer the interested reader to the article \cite{hardy_harris:spine_approach_applications}.

We first embellish our probability space by keeping track of some extra information about one particular infinite line of descent or \emph{spine}. This line of descent is defined as follows: our original particle is part of the spine; when this particle dies, we choose one of its offspring uniformly at random to become part of the spine. We continue in this manner: when a spine particle dies, we choose uniformly at random between its offspring to decide which becomes part of the spine. In this way at any time $t\geq0$ we have exactly one particle in $N(t)$ that is part of the spine. We refer to both this particle and its position with the label $\xi_t$; this is an abuse of notation, but it should always be clear from the context which meaning is intended. It is not hard to see that the spatial motion of the spine, $(\xi_t)_{t\geq0}$, is a standard Brownian motion.

The resulting probability measure (on the set of \emph{marked Galton-Watson trees with spines}) we denote by $\Pt$, and we find need for four different filtrations to encode differing amounts of this new information:
\begin{itemize}
\item{$\Fg_t$ contains the all the information about the marked tree up to time $t$. However, it does not know which particle is the spine at any point. Thus it is simply the natural filtration of the original branching Brownian motion.}
\item{$\Ft_t$ contains all the information about both the marked tree and the spine up to time $t$.}
\item{$\Gt_t$ contains all the information about the spine up to time $t$, including the birth times of other particles along its path, and how many particles were born at each of these times; it does not know anything about the rest of the tree.}
\item{$\Gg_t$ contains just the spatial information about the spine up to time $t$; it does not know anything about the rest of the tree.}
\end{itemize}
We note that $\Fg_t \subseteq \Ft_t$ and $\Gg_t\subseteq\Gt_t\subseteq\Ft_t$, and also that $\Pt$ is an extension of $\Pb$ in that $\Pt |_{\Fg_\infty} = \Pb$. All of the above is covered more rigourously in \cite{hardy_harris:spine_approach_applications}.

\begin{lem}[Many-to-one lemma]
If $g(t)$ is $\Gg_t$-measureable and can be written
\[g(t) = \sum_{u\in N(t)} g_u(t)\ind_{\{\xi_t = u\}}\]
where each $g_u(t)$ is $\Fg_t$-measureable, then
\[\Eb\left[\sum_{u\in N(t)} g_u(t)\right] = e^{rmt}\tilde\Eb[g(t)].\]
\end{lem}

This lemma is extremely useful as it allows us to reduce questions about the entire population down to calculations involving just one standard Brownian motion --- the spine. A proof may be found in \cite{hardy_harris:spine_approach_applications}.

\subsection{Martingales and changes of measure}
For $f\in C[0,1]$ and $\theta\in[0,1]$ define
\[N_T(f,\varepsilon,\theta) := \{u\in N(\theta T) : |X_u(t)-Tf(t/T)|<\varepsilon T \hs \forall t\in[0,\theta T]\}\]
so that $N_T(f,\varepsilon,\theta) = N_T(B(f,\varepsilon),\theta)$. We look for martingales associated with these sets. For convenience, in this section we use the shorthand
\[N_T(t):= N_T(f,\varepsilon,t/T).\]
Since the motion of the spine is simply a standard Brownian motion under $\Pt$, It\^{o}'s formula shows that for $t\in[0,T]$, the process
\[V_T(t) := e^{\pi^2 t/8\varepsilon^2 T^2} \cos\left(\frac{\pi}{2\varepsilon T}(\xi_t-Tf(t/T))\right)e^{\int_0^t f'(s/T) d\xi_s - \frac{1}{2}\int_0^t f'(s/T)^2 ds}\]
is a $\Gg_t$-martingale under $\Pt$. By stopping this process at the first exit time of the Brownian motion from the tube $\{(x,t): |Tf(t/T)-x|< \varepsilon T\}$, we obtain also that
\[\zeta_T(t) := V_T(t) \ind_{\{|Tf(s/T)-\xi_s| < \varepsilon T \hsl \forall s\leq t\}}\]
is a $\Gg_t$-martingale on $[0,T]$. As in \cite{hardy_harris:spine_approach_applications}, we may build from $\zeta_T$ a collection of $\Ft_t$-martingales $\tilde\zeta_T$ on $[0,T]$ given by
\[\tilde\zeta_T(t) := \prod_{v<\xi_t} A_v e^{-rmt}\zeta_T(t),\]
but these martingales will not be examined in this article --- they are important only in changing measure below, and in that when we project $\tilde\zeta_T(t)$ back onto $\Fg_t$ we get a new set of mean-one $\Fg_t$-martingales $Z_T$. These processes $Z_T$ are the main objects of interest in this section, and can be expressed for $t\in[0,T]$ as the sum
\[Z_T(t) = \sum_{u\in N_T(t)} V_T^{(u)}(t)e^{-rmt}\]
where
\small
\[V_T^{(u)}(t) := e^{\pi^2 t / 8\varepsilon^2 T^2} \cos\left(\frac{\pi}{2\varepsilon T}(X_u(t)-Tf(t/T))\right)e^{\int_0^t f'(s/T) dX_u(s) - \frac{1}{2}\int_0^t f'(s/T)^2 ds}.\]
\normalsize
We now define new measures, $\Qt_T$, via
\[\Qt_T|_{\Ft_t} = \tilde\zeta_T(t)\Pt|_{\Ft_t}\]
for $t\leq T$ --- and note that
\[\Qt_T|_{\Fg_t} = Z_T(t)\Pt|_{\Fg_t} \hs\hs \hbox{and} \hs\hs \Qt_T|_{\Gg_t} = \zeta_T(t)\Pt|_{\Gg_t}.\]

\begin{lem}\label{change_of_meas_lem}
Under $\Qt_T$, the spine $\xi$ moves as a Brownian motion with drift
\[f'(t/T) - \frac{\pi}{2\varepsilon T} \tan\left(\frac{\pi}{2\varepsilon T}(x - Tf(t/T))\right)\]
when at position $x$ at time $t$; in particular,
\[|\xi_t - Tf(t/T)|\leq\varepsilon T \hs \forall t\leq T \hs \Qt_T\hbox{-almost surely.}\]
Each particle $u$ in the spine dies at an accelerated rate $(m+1)r$, to be replaced by a random number $A_u$ of offspring where $A_u$ is taken from the \emph{size-biased} distribution relative to $A$, given by $\Qt_T(A_u = k) = kP(A=k)(m+1)^{-1}$, $k=0,1,\ldots$ (note that this distribution does not depend on $T$). All other particles, once born, behave \emph{exactly as they would under $\Pb$}: they move like independent standard Brownian motions, die at the usual rate $r$, and give birth to a number of particles that is distributed like $A$.
\end{lem}

\begin{proof}
Most of this is standard in the spine literature; for example proof can be found in \cite{hardy_harris:spine_approach_applications}. We will not use the precise drift of the spine except for the fact that the spine remains within the tube: to see this note that since the event is $\Gt_T$-measurable,
\[\Qt_T(\exists t\leq T : |\xi_t - Tf(t/T)|>\varepsilon T) = \tilde\Eb[\zeta_T(T)\ind_{\{\exists t\leq T : |\xi_t - Tf(t/T)|>\varepsilon T\}}] = 0\]
by the definition of $\zeta_T(T)$.
\end{proof}

\noindent
Another important tool in this section is the \emph{spine decomposition}.

\begin{lem}[Spine decomposition]
$\Qt_T$-almost surely,
\[\Qt_T[Z_T(t)|\tilde\Gg_T] = \sum_{u<\xi_t} (A_u-1) V_T(S_u)e^{-rmS_u} + V_T(t)e^{-rmt}\]
where we recall that $\{u<\xi_t\}$ is the set of ancestors of the spine particle at time $t$, and $S_u$ denotes the time at which particle $u$ split into two new particles.
\end{lem}

\noindent
A proof of the spine decomposition may be found in \cite{hardy_harris:spine_approach_applications}.

\begin{lem}
\label{int_lem}
If $f\in C^2[0,1]$ then for any $u\in N_T(t)$, almost surely under both $\Pt$ and $\Qt_T$ we have
\[\left|\int_0^t f'(s/T) dX_u(s) - \int_0^t f'(s/T)^2 ds\right| \leq 2\varepsilon T\int_0^{t/T} |f''(s)| ds + \varepsilon T|f'(0)|.\]
\end{lem}

\begin{proof}
From the integration by parts formula for It\^o calculus (since for any particle $u\in N(t)$, $(X_u(s),0\leq s\leq t)$ is a Brownian motion under $\Pt$) we know that for any $g\in C^2[0,1]$ with $g(0)=0$, under $\Pt$,
\[g'(t)X_u(t) = \int_0^t g''(s)X_u(s) ds + \int_0^t g'(s)dX_u(s).\]
From ordinary integration by parts,
\[\int_0^t g'(s)^2 ds = g'(t)g(t) - \int_0^t g(s)g''(s) ds.\]
Now set $g(t)=Tf(t/T)$ for $t\in[0,T]$. We note that, if $u\in N_T(t)$ then $|X_u(s) - g(s)|< \varepsilon T$ for all $s\leq t$. Thus
\begin{align*}
&\biggl|\int_0^t f'(s/T) dX_u(s) - \int_0^t f'(s/T)^2 ds\hspace{0.5mm}\biggr|\\
&= \biggl|\int_0^t g'(s) dX_u(s) - \int_0^t g'(s)^2 ds\hspace{0.5mm}\biggr|\\
&\leq \biggl|\hspace{0.5mm}g'(t)(X_u(t)-g(t)) - \int_0^t g''(s)(X_u(s)-g(s))ds\hspace{0.5mm}\biggr|\\
&\leq |g'(t) - g'(0)| \varepsilon T + |g'(0)|\varepsilon T - \int_0^t |g''(s)| \varepsilon T ds\\
&\leq 2\varepsilon T\int_0^{t} |g''(s)| ds + \varepsilon T|g'(0)|\\
&= 2\varepsilon T\int_0^{t/T} |f''(s)|ds + \varepsilon T|f'(0)|
\end{align*}
almost surely under $\Pt$ and, since $\Qt_T\ll \Pt$, almost surely under $\Qt_T$.
\end{proof}

\noindent
We now use this result to give approximations on $Z_T(t)$ under certain conditions. One of these conditions involves the seemingly unnatural assumption $f'(0)=0$. This is caused by the fact that in this section we make no approximations to the path of the spine under $\Qt_T$ except for using that it always remains within $\varepsilon T$ of our $T$-rescaled path --- hence we are left with a rather bad estimate on its path at small times, where it will not get anywhere near $\varepsilon T$. This does not matter to us, however, precisely because of this freedom to move within the $\varepsilon$-tube about $f$: if $f'(0)\neq0$ then we may choose $g$ near to $f$ (in an appropriate way; certainly within the $\varepsilon$-tube) such that $g'(0)=0$. This issue arises in Lemma \ref{bound_prob_below_1_lem} and rigorous details are given there.

\begin{lem}\label{spine_domination}
If $f\in C^2[0,1]$, $f'(0)=0$ and $rm\phi > \frac{1}{2}\int_0^\phi f'(s)^2 ds$ for all $\phi\in(0,\theta]$, then for small enough $\varepsilon>0$ and any $T>0$ and $t\leq\theta T$, there exists $\eta>0$ such that
\[\Qt_T[Z_T(t)|\Gt_T] \leq \sum_{u<\xi_T} (A_u-1) e^{\pi^2/8\varepsilon^2 T - \eta S_u} + e^{\pi^2/8\varepsilon^2 T - \eta t}\]
$\Qt_T$-almost surely.
\end{lem}

\begin{proof}
Since $rm\phi > \frac{1}{2}\int_0^\phi f'(s)^2 ds$ for all $\phi\in(0,\theta]$ and $f'(0)=0$, we may choose $\eta>0$ such that
\[2\eta\phi \leq rm\phi - \frac{1}{2}\int_0^\phi f'(s)^2 ds \hs \forall\phi\in[0,\theta].\]
Then for any $\varepsilon>0$ satisfying
\[2\varepsilon\int_0^\phi |f''(s)|ds \leq \eta\phi \hs \forall\phi\in[0,\theta]\]
we have, by Lemma \ref{int_lem} (since $f'(0)=0$ and using the fact that under $\Qt_T$ the spine is always in $N_T(t)$),
\[V_T(t)e^{-rmt} \leq e^{\pi^2/8\varepsilon^2 T - rmt + \frac{T}{2}\int_0^{t/T}f'(s)^2 ds + 2\varepsilon T\int_0^{t/T}|f''(s)|ds} \leq e^{\pi^2/8\varepsilon^2 T - \eta t}\]
for all $t\in[0,\theta T]$. Plugging this into the spine decomposition, we get
\[\Qt_T[Z_T(t)|\Gt_T] \leq \sum_{u<\xi_T} (A_u-1) e^{\pi^2/8\varepsilon^2 T - \eta S_u} + e^{\pi^2/8\varepsilon^2 T - \eta t}. \qedhere\]
\end{proof}

\begin{prop}\label{UI}
If $f\in C^2[0,1]$, $f'(0)=0$ and $rm\phi > \frac{1}{2}\int_0^\phi f'(s)^2 ds$ for all $\phi\in(0,\theta]$, then for small enough $\varepsilon>0$ the set $\{Z_T(t) : T\geq 1, t\leq\theta T\}$ is uniformly integrable under $\Pb$.
\end{prop}

\begin{proof}
Fix $\delta>0$. We first claim that there exists $K$ such that
\[\sup_{T\geq1,\hsl t\leq\theta T}\Qt_T(\Qt_T[Z_T(t)|\Gt_T] > K) < \delta/2.\]
To see this, take an auxiliary probability space with probability measure $Q$, and on this space consider a sequence $A_1, A_2, \ldots$ of independent and identically distributed random variables satisfying
\[Q(A_i=k) = \frac{k\Pb(A=k)}{m+1}\]
so that the $A_i$ have the same distribution as births $A_u$ along the spine under $\Qt_T$ (recall that there is no dependence on $T$). Take also a sequence $e_1, e_2, \ldots$ of independent random variables that are exponentially distributed with parameter $r(m+1)$; then set $S_n = e_1 + \ldots + e_n$ (so that the random variables $S_n$ have the same distribution as the birth times along the spine under $\Qt_T$). By Lemma \ref{spine_domination} we have
\[\sup_{\substack{T\geq1\\ t\leq\theta T}}\Qt_T(\Qt_T[Z_T(t)|\Gt_T] > K) \leq Q\left(\sum_{j=1}^\infty (A_j-1) e^{\pi^2/8\varepsilon^2 - \eta S_j} + e^{\pi^2/8\varepsilon^2} > K\right).\]
Hence our claim holds if the random variable $\sum_{j=1}^\infty (A_j-1) e^{- \eta S_j}$ can be shown to be $Q$-almost surely finite. Now for any $\gamma\in(0,1)$,
\begin{align*}
Q(\sum_n (A_n-1) e^{-\eta S_n} = \infty) &\leq Q(A_n e^{-\eta S_n} > \gamma^n \hbox{ infinitely often})\\
&\leq Q\left(\frac{\log A_n}{n} > \log \gamma + \frac{\eta S_n}{n} \hbox{ infinitely often}\right).
\end{align*}
By the strong law of large numbers, $S_n/n \to 1/r(m+1)$ almost surely under $Q$; so if $\gamma\in(\exp(-\eta/r(m+1)),1)$ then the quantity above is no larger than
\[Q\left(\limsup_{n\to\infty} \frac{\log A_n}{n} > 0\right).\]
But this quantity is zero by Borel-Cantelli: indeed, for any $T$,
\begin{multline*}
\sum_n Q\left(\frac{\log A_n}{n}>\varepsilon\right) = \sum_n Q(\log  A_1 > \varepsilon n )\\
\leq \int_0^\infty Q(\log  A_1 \geq \varepsilon x) dx = Q\left[\frac{\log  A_1}{\varepsilon}\right]
\end{multline*}
which is finite for any $\varepsilon>0$ since (by direct calculation from the distribution of $A_1$ under $Q$) $Q[\log  A_1] = \Pt[A \log  A]<\infty$ (this was one of our assumptions at the beginning of the article). Thus our claim holds.

Now choose $M>0$ such that $1/M < \delta/2$; then for $K$ chosen as above, and any $T\geq1$, $t\leq\theta T$,
\begin{align*}
\Qt_T(Z_T(t) > MK) &\leq \Qt_T( Z_T(t) > MK, \hsl \Qt_T[Z_T(t)|\Gt_T]\leq K )\\
&\hspace{45mm} + \Qt_T(\Qt_T[Z_T(t)|\Gt_T] > K )\\
&\leq \Qt_T\left[\frac{Z_T(t)}{MK}\ind_{\{\Qt_T[Z_T(t)|\Gt_T] \leq K\}}\right] + \delta/2\\
&= \Qt_T\left[\frac{\Qt_T[Z_T(t)|\Gt_T]}{MK}\ind_{\{\Qt_T[Z_T(t)|\Gt_T] \leq K\}}\right] + \delta/2\\
&\leq 1/M + \delta/2 \leq \delta.
\end{align*}
Thus, setting $K'=MK$, for any $T\geq1$, $t\leq\theta T$,
\[\Pb[Z_T(t)\ind_{\{Z_T(t)>K'\}}] = \Qt_T(Z_T(t)>K') \leq \delta.\]
Since $\delta>0$ was arbitrary, the proof is complete.
\end{proof}

As our final result in this section we link explicitly the martingales $Z_T$ with the number of particles $N_T$.

\begin{lem}\label{int_by_parts_lem}
For any $\delta>0$, if $f\in C^2[0,1]$, $f(0)=0$ and $\varepsilon$ is small enough then
\[Z_T(\theta T) \leq |N_T(f,\varepsilon,\theta)|\exp\left(\frac{\pi^2\theta}{8\varepsilon^2 T} - rm\theta T + \frac{T}{2}\int_0^\theta f'(s)^2 ds + \delta T \right).\]
\end{lem}

\begin{proof}
Simply plugging the result of Lemma \ref{int_lem} into the definition of $Z_T(\theta T)$ gives the desired inequality.
\end{proof}

We note here that, in fact, a similar bound can be given in the opposite direction, so that $N_T(f,\varepsilon/2,\theta)$ is dominated by $Z_T(\theta T)$ multiplied by some deterministic function of $T$. We will not need this bound, but it is interesting to note that the study of the martingales $Z_T$ is in a sense equivalent to the study of the number of particles $N_T$.

\section{The lower bound}\label{lower_bound_section}

\subsection{The heuristic for the lower bound}
We want to show that $N_T(f,\varepsilon,\theta)$ cannot be too small for large $T$. For $f\in C[0,1]$ and $\theta\in[0,1]$, define
\[J(f,\theta):=\left\{\begin{array}{ll}rm\theta - \frac{1}{2}\int_0^\theta f'(s)^2 ds & \hbox{ if } f\in H_1\\ -\infty & \hbox{ otherwise.}\end{array}\right.\]
We note that $J$ resembles our rate function $K$, but without the truncation at the extinction time $\theta_0$. We shall work mostly with the simpler object $J$, before deducing our result involving $K$ at the very last step. We now give a short heuristic to describe our route through the proof of the lower bound.

\vspace{3mm}

\emph{Step 1.} Consider a small time $\eta T$. How many particles are in $N_T(f,\varepsilon,\eta)$? If $\eta$ is much smaller than $\varepsilon$, then (with high probability) no particle has had enough time to reach anywhere near the edge of the tube (approximately distance $\varepsilon T$ from the origin) before time $\eta T$. Thus, with high probability,
\[|N_T(f,\varepsilon,\eta)| = |N(\eta T)| \approx \exp(rm\eta T).\]

\emph{Step 2.} Given their positions at time $\eta T$, the particles in $N_T(f,\varepsilon,\eta)$ act independently. Each particle $u$ in this set thus draws out an independent branching Brownian motion. Let $N_T(u,f,\varepsilon,\theta)$ be the set of descendants of $u$ that are in $N_T(f,\varepsilon,\theta)$. How big is this set? Since $\eta$ is very small, each particle $u$ is close to the origin. Thus we may hope to find some $q<1$ such that
\[\Pb\left(|N_T(u,f,\varepsilon,\theta)| < \exp(J(f,\theta)T - \delta T)\right) \leq q.\]
(Of course, in reality we believe that this quantity will be exponentially small --- but to begin with, the constant bound can be shown more readily.)

\vspace{2mm}

\emph{Step 3.} If $N_T(f,\varepsilon,\theta)$ is to be small, then each of the sets $N_T(u,f,\varepsilon,\theta)$ for $u\in N_T(f,\varepsilon,\eta)$ must be small. Thus
\[\Pb\left(|N_T(f,\varepsilon,\theta)| < \exp(J(f,\theta)T-\delta T)\right) \lesssim q^{\hsl\exp(rm\eta T)},\]
and we may apply Borel-Cantelli to deduce our result along lattice times (that is, times $T_j$, $j\geq0$ such that there exists $\tau>0$ with $T_j - T_{j-1} = \tau$ for all $j\geq1$).

\emph{Step 4.} We carry out a simple tube-reduction argument to move to continuous time.  The idea here is that if the result were true on lattice times but not in continuous time, the number of particles in $N_T(f,\varepsilon,\theta)$ must fall dramatically at infinitely many non-lattice times. We simply rule out this possibility using standard properties of Brownian motion.

\vspace{3mm}

The most difficult part of the proof is Step 2. However, the spine results of Section \ref{spine_section} will simplify our task significantly.

\subsection{The proof of the lower bound}

We begin with Step 1 of our heuristic, considering the size of $N_T(f,\varepsilon,\eta)$ for small $\eta$.

\begin{lem}\label{rightmost_lem}
For any continuous $f$ with $f(0)=0$ and any $\varepsilon>0$, there exist $\eta>0$, $k>0$ and $T_1$ such that
\[\Pb(\exists u\in N(\eta T) : u\not\in N_T(f,\varepsilon/2,\eta)) \leq e^{-kT} \hs \hs \forall T\geq T_1.\]
\end{lem}

\begin{proof}
Choose $\eta$ small enough that $\sup_{s\in[0,\eta]} |f(s)|<\varepsilon/4$. Then, using the many-to-one lemma and standard properties of Brownian motion,
\begin{align*}
\Pb&(\exists u\in N(\eta T) : u\not\in N_T(f,\varepsilon/2,\eta))\\
& = \Pb\left(\exists u\in N(\eta T) : \sup_{s\leq\eta}|X_u(sT) - Tf(s)|\geq \varepsilon T/2\right)\\
& \leq \Pb\left[\sum_{u\in N(\eta T)}\ind_{\{\sup_{s\leq\eta}|X_u(sT) - Tf(s)|\geq \varepsilon T/2\}}\right]\\
& \leq e^{rm\eta T} \Pt\left(\sup_{s\leq\eta}|\xi_{sT} - Tf(s)| \geq \varepsilon T/2\right)\\
& \leq e^{rm\eta T} \Pt\left(\sup_{s\leq\eta}|\xi_{sT}|\geq \varepsilon T/4\right)\\
&\leq \frac{16\sqrt\eta e^{rm\eta T - \varepsilon^2 T/32\eta}}{\varepsilon\sqrt{2\pi T}}.
\end{align*}
A suitably small choice of $\eta$ gives the exponential decay required.
\end{proof}

We now move on to Step 2, using the results of Section \ref{spine_section} to bound the probability of having a small number of particles strictly below 1. The bound given is extremely crude, and there is much room for manoeuvre in the proof, but any improvement would only add unnecessary detail.

\begin{lem}\label{bound_prob_below_1_lem}
If $f\in C^2[0,1]$ and $J(f,s)>0$ $\forall s\in(0,\theta]$, then for any $\varepsilon>0$ and $\delta>0$ there exists $T_0\geq0$ and $q<1$ such that
\[\Pb\left(|N_T(f,\varepsilon,\theta)|<e^{J(f,\theta)T-\delta T}\right)\leq q \hs\hs \forall T\geq T_0.\]
\end{lem}

\begin{proof}
Note that by Lemma \ref{int_by_parts_lem} for small enough $\varepsilon>0$ and large enough $T$,
\[|N_T(f,\varepsilon,\theta)|e^{-J(f,\theta)T+\delta T/2} \geq Z_T(\theta T)\]
and hence
\[\Pb\left(|N_T(f,\varepsilon,\theta)| < e^{J(f,\theta)T-\delta T}\right) \leq \Pb\left(Z_T(\theta T) < e^{-\delta T/2}\right).\]
Suppose first that $f'(0)=0$. Then, again for small enough $\varepsilon$, by Proposition \ref{UI} the set $\{Z_T(\theta T), T\geq 1, t\in[1,\theta T]\}$ is uniformly integrable. Thus we may choose $K$ such that
\[\sup_{T\geq1}\Eb[Z_T(\theta T) \ind_{\{Z_T(\theta T) > K\}}] \leq 1/4,\]
and then
\begin{align*}
1 = \Eb[Z_T(\theta T)] &= \Eb[Z_T(\theta T)\ind_{\{Z_T(\theta T)\leq 1/2\}}] + \Eb[Z_T(\theta T)\ind_{\{1/2< Z_T(\theta T)\leq K\}}]\\
&\hspace{20mm} + \Eb[Z_T(\theta T)\ind_{\{Z_T(\theta T) > K\}}]\\
&\leq 1/2 + K\Pb(Z_T(\theta T) > 1/2 ) + 1/4
\end{align*}
so that
\[\Pb(Z_T(\theta T) > 1/2) \geq 1/4K.\]
Hence for large enough $T$,
\[\Pb\left(|N_T(f,\varepsilon,\theta)| < e^{J(f,\theta)T-\delta T}\right) \leq 1 - 1/4K.\]
This is true for all small $\varepsilon>0$; but increasing $\varepsilon$ only increases $|N_T(f,\varepsilon,\theta)|$ so the statement holds for all $\varepsilon>0$. Finally, if $f'(0)\neq0$ then choose $g\in C^2[0,\theta]$ such that $g(0)=g'(0)=0$, $\sup_{s\leq\theta}|f-g|\leq\varepsilon/2$, $J(g,\phi)>0$ for all $\phi\leq\theta$ and $J(g,\theta)> J(f,\theta)-\delta/2$ (for small $\eta$, the function
\[g(t):= \left\{\begin{array}{ll}f(t) + at + bt^2 +ct^3 +dt^4 & \hbox{if }t\in[0,\eta)\\ f(t) & \hbox{if } t\in[\eta,1]\end{array}\right.\]
will work for suitable $a,b,c,d\in\Rb$). Then as above we may choose $K$ such that
\[\Pb(|N_T(f,\varepsilon,\theta)| < e^{J(f,\theta)T-\delta T}) \leq \Pb(|N_T(g,\varepsilon/2,\theta)| < e^{J(g,\theta)T-\delta T/2}) \leq 1-1/4K\]
as required.\end{proof}

Our next result runs along integer times --- these times are sufficient for our needs, although the following proof would in fact work for any lattice times.

\begin{prop}\label{lower_prop}
Suppose that $f\in C^2[0,1]$ and $J(f,s)>0$ $\forall s\in(0,\theta]$. Then
\[\liminf_{\substack{j\to\infty\\j\in\mathbb{N}}}\frac{1}{j}\log |N_j(f,\varepsilon,\theta)|\geq J(f,\theta)\]
almost surely.
\end{prop}

\begin{proof}
For any particle $u$, define
\begin{align*}
N_T(u,f,\varepsilon,\theta) &:= \{v\in N(\theta T) : u\leq v, \hs |X_v(t)-Tf(t/T)|<\varepsilon T \hs \forall t\in[0,\theta T]\}\\
& \hsl = \{v : u\leq v\} \cap N_T(f,\varepsilon,\theta),
\end{align*}
the set of descendants of $u$ that are in $N_T(f,\varepsilon,\theta)$.
Then for $\delta >0$ and $\eta \in[0,\theta]$,
\begin{align*}
\Pb&\left(\left.|N_T(f,\varepsilon,\theta)| < e^{J(f,\theta)T-\delta T} \right| \Fg_{\eta T}\right)\\
&\leq \prod_{u\in N_T(f,\varepsilon/2,\eta)}\Pb\left(\left.|N_T(u,f,\varepsilon,\theta)|<e^{J(f,\theta)T - \delta T} \right| \Fg_{\eta T}\right)\\
&\leq \prod_{u\in N_T(f,\varepsilon/2,\eta)}\Pb\left(|N_T(g,\varepsilon/2,\theta-\eta)|<e^{J(f,\theta)T-\delta T}\right)
\end{align*}
since $\{|N_T(u,f,\varepsilon,\theta)| : u\in N_T(f,\varepsilon/2,\eta)\}$ are independent random variables, and where $g:[0,1]\to\Rb$ is any twice continuously differentiable extension of the function
\[\begin{array}{rrcl} \bar g: & [0,\theta-\eta] &\to& \Rb\\
                          & t               &\to& f(t+\eta)-f(\eta).\end{array}\]
If $\eta$ is small enough, then
\[|J(f,\theta) - J(g,\theta - \eta)|<\delta/2\]
and
\[J(g,s) > 0 \hs\hs \forall s\in(0,\theta-\eta].\]
Hence, applying Lemma \ref{bound_prob_below_1_lem}, there exists $q<1$ such that for all large $T$,
\begin{multline*}
\Pb\left(|N_T(g,\varepsilon/2,\theta-\eta)|< e^{J(f,\theta)T-\delta T}\right)\\
\leq \Pb\left(|N_T(g,\varepsilon/2,\theta-\eta)| < e^{J(g,\theta-\eta)T-\delta T/2}\right) \leq q.
\end{multline*}
Thus for large $T$,
\begin{equation}\label{cond_prob}
\Pb\left(\left.|N_T(f,\varepsilon,\theta)| < e^{J(f,\theta)T-\delta T} \right| \Fg_{\eta T}\right) \leq q^{|N_T(f,\varepsilon/2,\eta)|}.
\end{equation}
Now, recalling that $N(t)$ is the \emph{total} number of particles alive at time $t$, it is well-known (and easy to calculate) that for $\alpha\in(0,1)$,
\[\Eb\left[\alpha^{|N(t)|}\right] \leq \frac{\alpha}{\alpha + (1-\alpha)e^{rt}}\]
(in fact this is exactly $\Eb[\alpha^{|N(t)|}]$ in the case of strictly dyadic branching).
Taking expectations in (\ref{cond_prob}), and then applying Lemma \ref{rightmost_lem}, for small $\eta$ we can get
\begin{align*}
\Pb&\left(|N_T(f,\varepsilon,\theta)| < e^{J(f,\theta)T-\delta T} \right)\\
&\leq \Pb\left(\exists u\in N(\eta T) : u\not\in N_T(f,\varepsilon/2,\eta)\right) + \Eb\left[q^{|N(\eta T)|}\right]\\
&\leq e^{-kT} + \frac{q}{q + (1-q)e^{r\eta T}}
\end{align*}
for some $k>0$ and all large enough $T$. The Borel-Cantelli lemma now tells us that
\[\Pb\left(\liminf_{j\to\infty}\frac{1}{j}\log|N_j(f,\varepsilon,\theta)| < J(f,\theta)-\delta \right) = 0,\]
and taking a union over $\delta>0$ gives the result.
\end{proof}

We note that our estimate on $\Eb[\alpha^{|N(t)|}]$ may not hold if we allowed the possibility of death with no offspring. In this case a more sophisticated estimate is required, taking into account the probability that the process becomes extinct.

We look now at moving to continuous time using Step 4 of our heuristic. For simplicity of notation, we break with convention by defining
\[\|f\|_\theta:=\sup_{s\in[0,\theta]}|f(s)|\]
for $f\in C[0,\theta]$ or $f\in C[0,1]$ (on this latter space, $\|\cdot\|_\theta$ is not a norm, but this will not matter to us).

\begin{prop}\label{tube_reduction_prop}
Suppose that $f\in C^2[0,1]$ and $J(f,s)>0$ $\forall s\in(0,\theta]$. Then
\[\liminf_{T\to\infty}\frac{1}{T}\log |N_T(f,\varepsilon,\theta)|\geq J(f,\theta)\]
almost surely.
\end{prop}

\begin{proof}
We claim first that for large enough $j\in\mathbb{N}$,
\begin{multline*}
\left\{|N_j(f,\varepsilon,\theta)|>\inf_{t\in[j,j+1]}|N_t(f,2\varepsilon,\theta)|\right\} \\
 \subseteq \left\{\exists u\in N(\theta (j+1)) : \sup_{t\in[j,j+1]} |X_u(t)-X_u(j)| > \frac{\varepsilon j }{ 2}\right\}.
\end{multline*}
Indeed, if $v\in N_j(f,\varepsilon,\theta)$, $t\in[j,j+1]$ and $s\in[0,\theta t]$ then for any descendant $u$ of $v$ at time $\theta t$, 
\begin{align*}
|X_u(s) - tf(s/t)| &\leq |X_u(s) - X_u(s\wedge \theta j)| + |X_u(s\wedge \theta j)-jf((s\wedge \theta j) / j)|\\
&\hspace{10mm} + |jf((s\wedge \theta j) / j) - jf(s/t)| + |jf(s/t) - tf(s/t)|\\
&\leq |X_u(s) - X_u(s\wedge \theta j)| + \varepsilon j\\
&\hspace{10mm} + j\sup_{\substack{x,y\in[0,\theta]\\|x-y|\leq 1/j}} |f(x)-f(y)| + \|f\|_\theta\\
&\leq |X_u(s) - X_u(s\wedge \theta j)| + \frac{3\varepsilon}{2}j \hs\hs \hbox{ for large j;}
\end{align*}
so that if any particle is in $N_j(f,\varepsilon,\theta)$ but not in $N_t(f,2\varepsilon,\theta)$ then it must satisfy
\[\sup_{j\leq s\leq t} |X_u(s) - X_u(j)| \geq \varepsilon j/ 2.\]
This is enough to establish the claim, and we deduce via the many-to-one lemma and standard properties of Brownian motion that
\begin{align*}
&\Pb(|N_j(f,\varepsilon,\theta)|>\inf_{t\in[j,j+1]}|N_t(f,2\varepsilon,\theta)|)\\
&\leq \Pb\left(\exists u\in N(\theta(j+1)) : \sup_{t\in[j,j+1]} |X_u(t) - X_u(j)| \geq \varepsilon j/2 \right)\\
& = e^{rm\theta(j+1)} \Pt(\sup_{t\in[j,j+1]} |\xi_t - \xi_j| \geq \varepsilon j/2 )\\
&\leq \frac{8}{\varepsilon j \sqrt{2\pi}}\exp(rm\theta(j+1) - \varepsilon^2 j^2/8).
\end{align*}
Since these probabilities are summable we may apply Borel-Cantelli to see that
\[\Pb(|N_j(f,\varepsilon,\theta)| > \inf_{t\in[j,j+1]}|N_t(f,2\varepsilon,\theta)| \hbox{ infinitely often} )=0.\]
Now,
\begin{multline*}
\Pb\left(\liminf_{T\to\infty}\frac{1}{T}\log |N_T(f,\varepsilon,\theta)| < J(f,\theta) \right)\\
\leq \Pb\left(\liminf_{j\to\infty}\frac{1}{j}\log |N_j(f,2\varepsilon,\theta)| < J(f,\theta)\right)\\
+ \Pb\left(\liminf_{j\to\infty} \frac{\inf_{t\in[j,j+1]}|N_t(f,\varepsilon,\theta)|}{|N_j(f,2\varepsilon,\theta)|} < 1 \right)
\end{multline*}
which is zero by Proposition \ref{lower_prop} and Borel-Cantelli.
\end{proof}

If we were including the possibility of death with no offspring then we would have to check that no particles in $N_j(f,\varepsilon,\theta)$ managed to reach the outside of the slightly altered $2\varepsilon$-tube and then die before time $j+1$. The only added difficulty would be in keeping track of notation.

We are now in a position to give our lower bound in full.

\begin{cor}\label{lower_bd}
For any open set $U\subseteq C[0,1]$ and $\theta\in[0,1]$, we have
\[\liminf_{T\to\infty}\frac{1}{T}\log|N_T(U,\theta)| \geq \sup_{f\in U}K(f,\theta)\]
almost surely.
\end{cor}

\begin{proof}
If $\sup_{f\in U}K(f,\theta)=-\infty$ then there is nothing to prove. Thus it suffices to consider the case when there exists $f\in U$ such that $\theta\leq\theta_0(f)$. Since $U$ is open, in this case we can in fact find $f\in U$ such that $J(f,s)>0$ for all $s\in(0,\theta]$ (if $J(f,\phi)=0$ for some $\phi\leq\theta$, just choose $\eta$ small enough that $(1-\eta)f \in U$) and such that $f$ is twice continuously differentiable on $[0,1]$ (twice continuously differentiable functions are dense in $C[0,1]$). Thus necessarily \mbox{$\sup_{g\in U}K(g,\theta)>0$}, and for any $\delta>0$ we may further assume (by a simple argument, for example by approximating with piecewise linear functions and then smoothing) that \mbox{$J(f,\theta)>\sup_{g\in U}K(g,\theta)-\delta$}. Again since $U$ is open, we may take $\varepsilon$ such that $B(f,\varepsilon)\subseteq U$; then clearly for any $T$
\[N_T(f,\varepsilon,\theta)\subseteq N_T(U,\theta)\]
so by Proposition \ref{lower_prop} we have
\[\liminf_{T\to\infty}\frac{1}{T}\log N_T(U,\theta) \geq \sup_{g\in U}K(g,\theta)-\delta\]
almost surely, and by taking a union over $\delta>0$ we may deduce the result.
\end{proof}

\section{The upper bound}\label{upper_bound_section}
Our plan is as follows: we first carry out the simple task of obtaining a bound along lattice times (Proposition \ref{upper_prop}). We then move to continuous time in Lemma \ref{upper_lattice_to_cts}, at the cost of restricting to open balls about fixed paths, by a tube-expansion argument similar to the tube-reduction argument used in Proposition \ref{tube_reduction_prop} of the lower bound. In Lemma \ref{extreme_lem} we then rule out the possibility of any particles following unusual paths, which allows us to restrict our attention to a compact set, and hence a finite number of small open balls about sensible paths. Finally we draw this work together in Proposition \ref{upper_bd_J} to give the bound in continuous time for any closed set $D$.

Our first task, then, is to establish an upper bound along integer times. As with the lower bound, these times are sufficient for our needs, although the following proof would work for any lattice times. In a slight abuse of notation, for $D\subseteq C[0,1]$ and $\theta\in[0,1]$ we define
\[J(D,\theta) := \sup_{f\in D} J(f,\theta).\]

\begin{prop}\label{upper_prop}
For any closed set $D\subseteq C[0,1]$ and $\theta\in[0,1]$ we have
\[\limsup_{\substack{j\to\infty\\j\in\mathbb{N}}}\frac{1}{j}\log|N_j(D,\theta)| \leq J(D,\theta)\]
almost surely.
\end{prop}

\begin{proof}
From the upper bound for Schilder's theorem (Theorem 5.1 of \cite{varadhan:large_devs_apps}) we have
\[\limsup_{T\to\infty}\frac{1}{T}\log\Pt(\xi_T \in N_T(D,\theta)) \leq - \inf_{f\in D} \frac{1}{2}\int_0^\theta f'(s)^2 ds.\]
Thus, by the many-to-one lemma,
\begin{align*}
\limsup_{T\to\infty}\frac{1}{T}\log\Eb\big[|N_T(D,\theta)|\big] &\leq \limsup_{T\to\infty} \frac{1}{T}\log\left(e^{rm\theta T}\Pt(\xi_T\in N_T(D,\theta))\right)\\
&\leq rm\theta - \inf_{f\in D}\frac{1}{2}\int_0^\theta f'(s)^2 ds\\
&= J(D,\theta).
\end{align*}
Applying Markov's inequality, for any $\delta>0$ we get
\[\limsup_{T\to\infty}\frac{1}{T}\log \Pb\big(|N_T(D,\theta)| \geq e^{J(D,\theta)T + \delta T}\big)\leq \limsup_{T\to\infty}\frac{1}{T}\log \frac{\Eb\big[|N_T(D,\theta)|\big]}{e^{J(D,\theta)T + \delta T}} \leq -\delta\]
so that
\[\sum_{j=1}^\infty \Pb\big(|N_j(D,\theta)| \geq e^{J(D,\theta)j + \delta j}\big) < \infty\]
and hence by the Borel-Cantelli lemma
\[\Pb\left(\limsup_{j\to\infty}\frac{1}{j}\log|N_j(D,\theta)| \geq J(D,\theta)+\delta\right)=0.\]
Taking a union over $\delta>0$ now gives the result.
\end{proof}

We note that the proof by Git \cite{git:almost_sure_path_properties} works up to this point; the rest of the proof of the upper bound will be concerned with plugging the gap in \cite{git:almost_sure_path_properties}.

For $D\subset C[0,1]$ and $\varepsilon>0$, let
\[D^\varepsilon := \{f\in C[0,1] : \inf_{g\in D}\|f-g\|\leq \varepsilon\}.\]
Recall that we defined $N_T(f,\varepsilon,\theta) := N_T(B(f,\varepsilon),\theta)$.

\begin{lem}\label{upper_lattice_to_cts}
If $D\subseteq C[0,1]$ and $f\in D$, then
\[\limsup_{T\to\infty}\frac{1}{T}\log|N_T(f,\varepsilon,\theta)|\leq J(D^{2\varepsilon},\theta)\]
almost surely.
\end{lem}

\begin{proof}
First note that
\begin{multline*}
\Pb\left(\limsup_{T\to\infty}\frac{1}{T}\log|N_T(f,\varepsilon,\theta)|>J(D^{2\varepsilon},\theta)+\delta\right)\\
\leq \Pb\left(\limsup_{j\to\infty}\frac{1}{j}\log|N_{j}(f,2\varepsilon,\theta)| > J(D^{2\varepsilon},\theta)\right)\\
+ \Pb\left(\limsup_{j\to\infty}\frac{1}{j}\log\sup_{t\in[j,j+1]}\frac{|N_t(f,\varepsilon,\theta)|}{|N_{j}(f,2\varepsilon,\theta)|} > \delta\right).
\end{multline*}
Since $f\in D$, the uniform closed ball of radius $2\varepsilon$ about $f$ is a subset of $D^{2\varepsilon}$, so by Proposition \ref{upper_prop},
\[\Pb\left(\limsup_{j\to\infty}\frac{1}{j}\log|N_{j}(f,2\varepsilon,\theta)| > J(D^{2\varepsilon},\theta)\right) = 0\]
and we may concentrate on the last term. We claim that for $j$ large enough, for any $t\in[j,j+1]$ we have
\[N_t(f,\varepsilon,\theta j/t) \subseteq N_j(f,2\varepsilon,\theta).\]
Indeed, if $u\in N_t(f,\varepsilon,\theta j/t)$ then for any $s\leq \theta j$,
\begin{align*}
&|X_u(s) - jf(s/j)|\\
&\leq \left|X_u(s) - t f\left(s/t\right)\right| + \left|jf\left(s/j\right) - t f\left(s/j\right)\right| + t\left|f\left(s/j\right) - f\left(s/t\right)\right|\\
&\leq t\varepsilon + \|f\|_\theta + t\sup_{\begin{subarray}{c}x,y\in[0,\theta]\\|x-y|\leq 1/j \end{subarray}}|f(x)-f(y)|
\end{align*}
which is smaller than $2\varepsilon j$ for large $j$ since $f$ is absolutely continuous.

We deduce that for large $j$ every particle in $N_t(f,\varepsilon,\theta)$ for any $t\in[j,j+1]$ has an ancestor in $N_j(f,2\varepsilon,\theta)$; thus, letting $N(u,s,t)$ be the set of all descendants (including, possibly, $u$ itself) of particle $u\in N(s)$ at time $t$,
\begin{align*}
&\Eb\left[\sup_{t\in[j,j+1]}\frac{|N_t(f,\varepsilon,\theta)|}{|N_j(f,2\varepsilon,\theta)|}\right]\\
&\leq \Eb\left[\frac{\Eb\left[\left.\sup_{t\in[j,j+1]}|N_t(f,\varepsilon,\theta)|\right|\Fg_{\theta j}\right]}{|N_j(f,2\varepsilon,\theta)|}\right]\\
&\leq \Eb\left[\frac{\Eb\left[\left.\sup_{t\in[j,j+1]}\sum_{u\in N_j(f,2\varepsilon,\theta)}|N(u,\theta j,\theta t)|\right|\Fg_{\theta j}\right]}{|N_j(f,2\varepsilon,\theta)|}\right].
\end{align*}
Since $|N(u,\theta j,\theta t)|$ is non-decreasing in $t$, using the Markov property we get
\begin{align*}
\Eb\left[\sup_{t\in[j,j+1]}\frac{|N_t(f,\varepsilon,\theta)|}{|N_j(f,2\varepsilon,\theta)|}\right] &\leq \Eb\left[\frac{\sum_{u\in N_j(f,2\varepsilon,\theta)}\Eb\big[|N(u,\theta j,\theta (j+1))|\big|\Fg_{\theta j}\big]}{|N_j(f,2\varepsilon,\theta)|}\right]\\
& = \Eb\left[\frac{|N_j(f,2\varepsilon,\theta)|\Eb[|N(\theta)|]}{|N_j(f,2\varepsilon,\theta)|}\right]\\
& = \exp(rm\theta).
\end{align*}
Hence by Markov's inequality
\[\Pb\left(\sup_{t\in[j,j+1]}\frac{|N_t(f,\varepsilon,\theta)|}{|N_j(f,2\varepsilon,\theta)|} > \exp\left(\delta j\right) \right) \leq \exp\left(rm\theta -\delta j\right)\]
and applying Borel-Cantelli
\[\Pb\left(\limsup_{j\to\infty}\frac{1}{j}\log\sup_{t\in[j,j+1]}\frac{|N_t(f,\varepsilon,\theta)|}{|N_{j}(f,2\varepsilon,\theta)|} > \delta\right)=0.\]
Again taking a union over $\delta>0$ gives the result.
\end{proof}

If we were considering the possibility of particles dying with no offspring then $N(u,\theta j,\theta t)$ would not be non-decreasing in $t$, but considering instead the set of all descendants of $u$ ever alive between times $\theta j$ and $\theta t$ would give us a slightly worse --- but still good enough --- estimate.

We move now onto ruling out extreme paths, by choosing a ``bad set'' $F_N$ and showing that no particles follow paths in this set. There is a balance to be found between including enough paths in $F_N$ that $C_0[0,1]\setminus F_N$ is compact, but not so many that we might find some (rescaled) Brownian paths within $F_N$ at large times.

For simplicity of notation, we extend the definition of $N_T(D,\theta)$ to sets \mbox{$D\subseteq C[0,\theta]$} in the obvious way, setting
\[N_T(D,\theta) := \{u \in N(\theta T) : \exists f \in D \hbox{ with } X_u(t) = Tf(t/T) \hs \forall t\in[0,\theta T]\}.\]

\begin{lem}\label{extreme_lem}
Fix $\theta\in[0,1]$. For $N\in\mathbb{N}$, let
\small
\[F_N:=\left\{f\in C[0,\theta] : \exists n\geq N, \hsl u,s\in[0,\theta] \hbox{ with } |u-s|\leq \frac{1}{n^2}, \hsl |f(u)-f(s)|>\frac{1}{\sqrt n}\right\}.\]
\normalsize
Then for large $N$
\[\limsup_{T\to\infty}\frac{1}{T}\log |N_T(F_N,\theta)| = -\infty\]
almost surely.
\end{lem}

\begin{proof}
Fix $T\geq S\geq0$; then for any $t\in[S,T]$,
\begin{align*}
\{\xi_t\in N_t(F_N,\theta)\} &= \left\{\exists n\geq N,\hsl u,s\in[0,\theta] : |u-s|\leq \frac{1}{n^2},\hsl \left|\frac{\xi_{ut} - \xi_{st}}{t}\right| > \frac{1}{\sqrt n}\right\}\\
&\subseteq \left\{\exists n\geq N,\hsl u,s\in[0,\theta] : |u-s|\leq \frac{1}{n^2},\hsl \left|\frac{\xi_{uT} - \xi_{sT}}{S}\right| > \frac{1}{\sqrt n}\right\}.
\end{align*}
Since the right-hand side does not depend on $t$, we deduce that
\begin{multline*}
\{\exists t\in[S,T] : \xi_t\in N_t(F_N,\theta)\} \\ \subseteq \left\{\exists n\geq N,\hsl u,s\in[0,\theta] : |u-s|\leq \frac{1}{n^2},\hsl \left|\frac{\xi_{uT} - \xi_{sT}}{S}\right| > \frac{1}{\sqrt n}\right\}.\end{multline*}
Now, for $s\in[0,\theta]$, define $\pi(n,s):=\lfloor 2n^2 s \rfloor /2n^2$. Suppose we have a continuous function $f$ such that \mbox{$\sup_{s\in[0,\theta]}|f(s) - f(\pi(n,s))|\leq 1/4\sqrt n$}. If $u,s\in[0,\theta]$ satisfy $|u-s|\leq 1/n^2$, then
\begin{align*}
&|f(u) - f(s)|\\
&\leq |f(u) - f(\pi(n,u))| + |f(s) - f(\pi(n,s))| + |f(\pi(n,s)) - f(\pi(n,u))|\\
&\leq \frac{1}{4\sqrt n} + \frac{1}{4\sqrt n} + \frac{2}{4\sqrt n} = \frac{1}{\sqrt n}.
\end{align*}
Thus
\[\{\exists t\in[S,T] : \xi_t \in N_t(F_N,\theta)\}\subseteq\left\{\exists n\geq N,\hsl s\leq\theta : \left|\frac{\xi_{sT}-\xi_{\pi(n,s)T}}{S}\right|>\frac{1}{4\sqrt n}\right\}.\]
Standard properties of Brownian motion now give us that
\begin{align*}
\Pt(\exists t\in[S,T] : \xi_t\in N_t(F_N,\theta)) &\leq \Pt\left(\exists n\geq N,\hsl s\leq\theta : |\xi_{sT}-\xi_{\pi(n,s)T}|>S/4\sqrt n\right)\\
&\leq \sum_{n\geq N} 2n^2 \Pt\left(\sup_{s\in[0,1/2n^2]}\left|\xi_{sT}\right|> S/4\sqrt n \right)\\
&\leq \sum_{n\geq N} \frac{8\sqrt{n^3 T}}{S\sqrt{\pi}}\exp\left(-\frac{S^2 n}{16T}\right).
\end{align*}
Taking $S=j$ and $T=j+1$, we note that for large $N$,
\[\sum_{n\geq N} \frac{8\sqrt{n^3 T}}{S\sqrt{\pi}}\exp\left(-\frac{S^2 n}{16T}\right) \leq \sum_{n\geq N} \exp\left(-\frac{j n}{32}\right) \leq \exp\left(-\frac{j N}{64}\right)\]
so that (again for large $N$),
\[\Pt(\exists t\in[j,j+1] : \xi_t\in N_t(F_N,\theta)) \leq \exp(-2rmj).\]
Applying Markov's inequality and the many-to-one lemma,
\begin{align*}
\Pb(\sup_{t\in[j,j+1]} |N_t(F_N,\theta)|\geq 1) &\leq \Eb\left[\sup_{t\in[j,j+1]} |N_t(F_N,\theta)|\right]\\
&\leq \Eb\left[\sum_{u\in N(j+1)} \ind_{\{\exists t\in[j,j+1], \hsl v\leq u \hsl:\hsl v\in N_t(F_N,\theta)\}} \right]\\
&\leq e^{rm\theta (j+1)}\Pt(\exists t\in[j,j+1] : \xi_t \in N_t(F_N,\theta))\\
&\leq \exp(rm\theta (j+1) - 2rmj).
\end{align*}
Thus, by Borel-Cantelli, we have that for large enough $N$
\[\Pb(\limsup_{j\to\infty} \sup_{t\in[j,j+1]} |N_t(F_N,\theta)| \geq 1)=0\]
and since $|N_T(F_N,\theta)|$ is integer-valued,
\[\limsup_{T\to\infty}\frac{1}{T}\log|N_T(F_N,\theta)| = -\infty\]
almost surely.
\end{proof}

Now that we have ruled out any extreme paths, we check that we can cover the remainder of our sets in a suitable way.

\begin{lem}\label{totally_bdd_lem}
For $\theta \in[0,1]$, let
\[C_0[0,\theta]:=\{f\in C[0,\theta] : f(0)=0\}.\]
For each $N\in\mathbb{N}$, the set $C_0[0,\theta]\setminus F_N$ is totally bounded under $\|\cdot\|_\theta$ (that is, it may be covered by open balls of arbitrarily small radius).
\end{lem}

\begin{proof}
Given $\varepsilon>0$ and $N\in\mathbb{N}$, choose $n$ such that \mbox{$n\geq N\vee (1/\varepsilon^2)$}. For any \mbox{$f\in C_0[0,\theta]\setminus F_N$}, if \mbox{$|u-s|<1/n^2$} then \mbox{$|f(u)-f(s)|\leq1/\sqrt n \leq \varepsilon$}. Thus \mbox{$C_0[0,\theta]\setminus F_N$} is equicontinuous (and, since each function must start from 0, uniformly bounded) and we may apply the Arzel\`a-Ascoli theorem to say that $C_0[0,\theta]\setminus F_N$ is relatively compact, which is equivalent to totally bounded since $(C[0,\theta], \|\cdot\|_\theta)$ is a complete metric space.
\end{proof}

We are now in a position to give an upper bound for any closed set $D$ in continuous time. This upper bound is not quite what we asked for in Theorem \ref{main_thm}, but this issue --- replacing $J$ with $K$ ---will be corrected in Corollary \ref{upper_bd}.

\begin{prop}\label{upper_bd_J}
If $D\subset C[0,1]$ is closed, then for any $\theta\in[0,1]$
\[\limsup_{T\to\infty}\frac{1}{T}\log|N_T(D,\theta)|\leq J(D,\theta)\]
almost surely.
\end{prop}

\begin{proof}
Clearly (since our first particle starts from 0) $N_T(D\setminus C_0[0,1],\theta)=\emptyset$ for all $T$, so we may assume without loss of generality that $D\subseteq C_0[0,1]$. Now, for each $\theta$,
\[f \mapsto \left\{\begin{array}{ll}\frac{1}{2}\int_0^\theta f'(s)^2 ds & \hbox{ if } f\in H_1\\ \infty & \hbox{ otherwise}\end{array}\right.\]
is a good rate function on $C_0[0,\theta]$ (that is, lower-semicontinuous with compact level sets): we refer to Section 5.2 of \cite{dembo_zeitouni:large_devs} but it is possible to give a proof by showing directly that the function is lower-semicontinuous, then applying Jensen's inequality and the Arzel\`a-Ascoli theorem to prove that its level sets in $C_0[0,1]$ are compact. Hence we know that for any $\delta>0$,
\[\{f\in C_0[0,\theta] : J(f,\theta) \geq J(D,\theta) + \delta\}\]
is compact, and since it is disjoint from
\[\{f\in C_0[0,\theta] : \exists g\in D \hbox{ with } f(s)=g(s) \hsl\forall s\in[0,\theta]\},\]
which is closed, there is a positive distance between the two sets. Thus we may fix $\delta>0$ and choose $\varepsilon>0$ such that $J(D^{2\varepsilon},\theta) < J(D,\theta)+\delta$. Then, by Lemma \ref{totally_bdd_lem}, for any $N$ we may choose a finite $\alpha$ (depending on $N$) and some $f_k$, $k=1,2,\ldots,\alpha$ such that balls of radius $\varepsilon$ about the $f_k$ cover $C_0[0,\theta]\setminus F_N$. Thus
\begin{multline*}
\Pb\left(\limsup_{T\to\infty}\frac{1}{T}\log|N_T(D,\theta)|>J(D,\theta)+\delta\right)\\
\leq \Pb\left(\limsup_{T\to\infty}\frac{1}{T}\log|N_T(F_N,\theta)| > J(D,\theta)+\delta\right)\\
+ \sum_{k=1}^\alpha \Pb\left(\limsup_{T\to\infty}\frac{1}{T}\log|N_T(f_k,\varepsilon,\theta)| > J(D^{2\varepsilon},\theta)\right).
\end{multline*}
By Lemma \ref{extreme_lem} and Lemma \ref{upper_lattice_to_cts}, for large enough $N$ the terms on the right-hand side are all zero. As usual we take a union over $\delta>0$ to complete the proof.
\end{proof}

\begin{cor}\label{upper_bd}
For any closed set $D\subseteq C[0,1]$ and $\theta\in[0,1]$, we have
\[\limsup_{T\to\infty}\frac{1}{T}\log|N_T(D,\theta)| \leq \sup_{f\in D} K(f,\theta)\]
almost surely.
\end{cor}

\begin{proof}
Since $|N_T(D,\theta)|$ is integer valued,
\[\frac{1}{T}\log |N_T(D,\theta)| < 0 \hs \Rightarrow \hs \frac{1}{T}\log |N_T(D,\theta)| = -\infty.\]
Thus, by Proposition \ref{upper_prop}, if $J(D,\theta)<0$ then
\[\Pb\left(\limsup_{T\to\infty}\frac{1}{T}\log|N_T(D,\theta)| > -\infty\right)=0.\]
Further, clearly for $\phi\leq\theta$ and any $T\geq0$, if $N_T(D,\phi)=\emptyset$ then necessarily $N_T(D,\theta)=\emptyset$. Thus if there exists $\phi\leq\theta$ with $J(D,\phi)<0$, then
\[\Pb\left(\limsup_{T\to\infty}\frac{1}{T}\log|N_T(D,\theta)| > -\infty\right)=0\]
which completes the proof.
\end{proof}

\noindent
Combining Corollary \ref{lower_bd} with Corollary \ref{upper_bd} completes the proof of Theorem \ref{main_thm}.

\section*{Appendix: The oversight in \cite{git:almost_sure_path_properties}}
In \cite{git:almost_sure_path_properties} it is written that under a certain assumption, setting
\[W_n = \left\{\omega\in\Omega : \limsup_{T\to\infty} \frac{1}{T} \log |N_T(D,\theta)| > J(D,\theta) + \frac{1}{n}\right\}\]
(it is not important what $J(D,\theta)$ is here) we have $\Pb(W_n)>0$ for some $n$. This is correct, but the article then goes on to say ``It is now clear that
\[\limsup_{T\to\infty} \frac{1}{T} \log \mathbb{E}\big[|N_T(D,\theta)|\big] \geq J(D,\theta) + \frac{1}{n}\hs\hbox{''}\]
which does not appear to be obviously true. To see this explicitly, work on the probability space $[0,1]$ with Lebesgue probability measure $\Pb$. Let $X_T$, $T\geq0$ be the c\`adl\`ag random process defined (for $\omega\in[0,1]$ and $T\geq0$) by
\[X_T(\omega) = \left\{\begin{array}{ll} e^{2T} & \hbox{if } \hsl T - n \in [\omega-e^{-4T}, \hsl \omega+e^{-4T}) \hbox{ for some } n\in\mathbb{N}\\
										e^{T}  & \hbox{otherwise.} \end{array}\right.\]
\begin{center}\includegraphics[height=3cm]{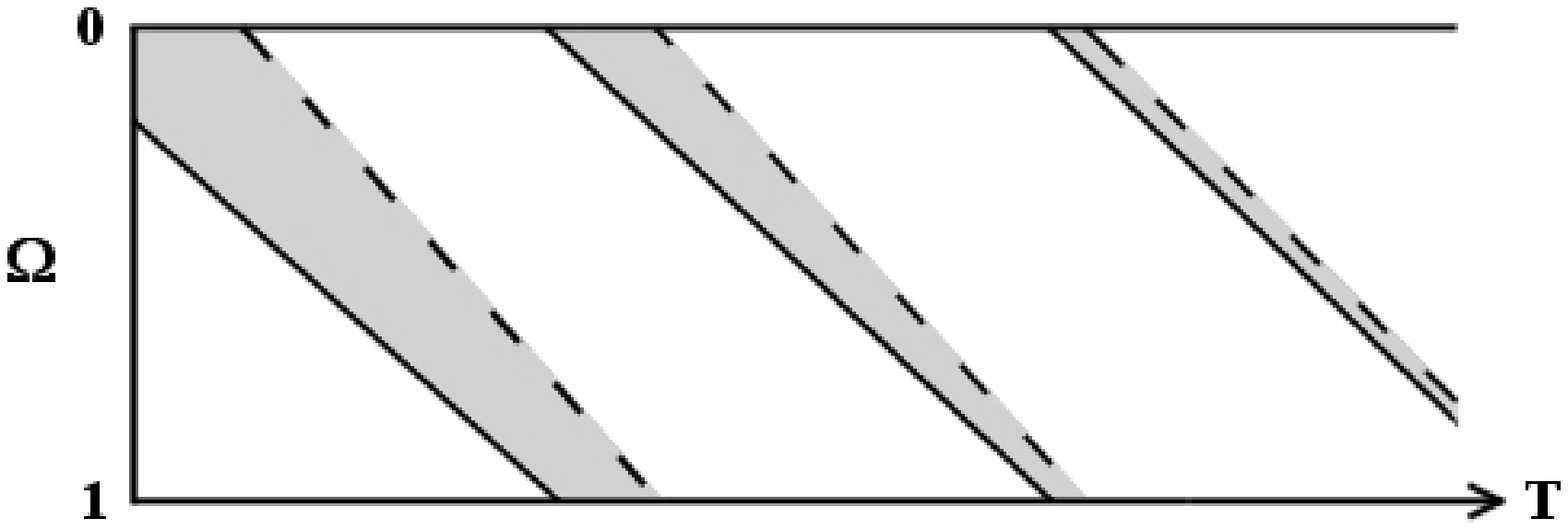}\end{center}
Then for every $\omega$,
\[\limsup \frac{1}{T}\log X_T(\omega) = 2\]
but
\[\frac{1}{T}\log\mathbb{E}[X_T] \to 1.\]

\bibliographystyle{plain}

\end{document}